\documentclass[oneside,english]{amsart}
\usepackage[T1]{fontenc}
\usepackage[latin9]{inputenc}
\usepackage{geometry}
\usepackage{verbatim}
\usepackage{amstext}
\usepackage{amsthm}
\usepackage{hyperref}

\makeatletter
\numberwithin{equation}{section}
\numberwithin{figure}{section}
\theoremstyle{plain}
\newtheorem{thm}{\protect\theoremname}
  \theoremstyle{plain}
  \newtheorem{lem}[thm]{\protect\lemmaname}
  \theoremstyle{plain}
  
  \theoremstyle{remark}
  \newtheorem{rem}[thm]{\protect\remarkname}

\makeatother

\usepackage{babel}
  \providecommand{\lemmaname}{Lemma}
  \providecommand{\remarkname}{Remark}
\providecommand{\theoremname}{Theorem}
\providecommand{\corollaryname}{Corollary}

\begin{document}

\title{Hyper $b$-ary expansions and Stern polynomials}

\author{Tanay Wakhare$^\ast$$^\dag$, Caleb Kendrick$^\ast$, Matthew Chung$^\ast$, Catherine Cassell$^\ast$, Stefano Santini$^\ast$, William Colin Mosley$^\ast$, Anand Raghu$^\ast$, Robert Morrison$^\ast$, Iman Schurman$^\ast$, Timothy Kevin Beal$^\ast$, Matthew Patrick$^\ast$}

\thanks{{\scriptsize
\hskip -0.4 true cm MSC(2010): Primary: 05A17; Secondary: 11B37.
$^\ast$Corresponding author
}}

\address{$^\ast$University of Maryland, College Park, MD 20742, USA}
\email{$^\dag$~twakhare@gmail.com}

\maketitle
\begin{abstract}
We study a recently introduced base $b$ polynomial analog of Stern's diatomic sequence, which generalizes Stern polynomials of Klavar, Dilcher, Ericksen, Mansour, Stolarsky, and others. We lift some basic properties of base $2$ Stern polynomials to arbitrary base, and introduce a matrix characterization of Stern polynomials. By specializing, we recover some new number theoretic results about hyper $b$-ary partitions, which count partitions of $n$ into powers of $b$.
\end{abstract}

\section{Introduction}
The aim of this note is to study the properties of arbitrary base polynomial analogs of the Stern sequence, recently introduced in \cite{Karl1}. These are intimately connected to both the theory of automatic sequences and binary partitions. The classical \textbf{Stern diatomic sequence} is defined by the two recurrences
\begin{align*}
s(2n) &= s(n),\\
s({2n+1}) &= s(n)+s(n+1),
\end{align*}
with the initial conditions $s(0)= 0, s(1) = 1$. The Stern sequence is closely linked to the theory of automatic sequences \cite{Allouche}; informally, an automatic sequence can be represented by a finite automaton, to which we feed in the digits of $n$ in sequential order. These sequences exhibit interesting aspects of both order and disorder, making them ideal for modeling semi-chaotic physical systems such as quasicrystals. 

For our purposes, the most important property of the Stern sequence is its combinatorial interpretation: $s({n+1})$ is equal to the number of \textbf{hyperbinary expansions} of $n$, which is the number of ways to write $n$ as the sum of powers of $2$, with each part used at most twice. While this may seem like a contrived definition, it is one of the simplest possible binary partition functions. While we typically consider partition functions built from sets of integers with nonzero asymptotic density, binary partitions are restricted partitions built from the parts $\{1,2,4,8,16,\ldots\}$. When each part can be used at most once, this is simply the base $2$ expansion of $n$. When we place no restriction on the number of parts, we obtain the binary partition function of Churchhouse \cite{Churchhouse}. In this paper, we consider Stern numbers, which allow each part to be used at most twice. Any progress on the Stern case will provide intuition for treating more complicated binary partition functions, which qualitatively behave very differently from more commonly studied partition functions.

We can then consider a base $b$ analog of the Stern sequence, such that the $(n+1)^{\text{th}}$ term equals the number of hyper $b$-ary expansions (also known as base $b$ overexpansions) of $n$. These are partitions from the set $\{1,b,b^2,b^3,\ldots\}$, such that each part can be used at most $b$ times. We will denote these by $s_b(n)$, and note that they are completely determined by the recurrence
\begin{align*}
&s_b({b(n-1)+j+1}) = s_b(n), \thinspace\thinspace\thinspace(1\leq j \leq b-1), \\
&s_b({bn+1}) =s_b(n)+s_b(n+1).
\end{align*}
The goal of this paper is to consider an even more general case, the \textbf{base $b$ Stern polynomials}, which are polynomials in the $b$ variables $\{z_1,\ldots,z_b\}$, and are defined by the following set of $b$ recurrences:
\begin{align*}
&w_T(b(n-1)+j+1|z_1\ldots,z_b) = z_jw_T(n|z_1^{t_1}\ldots,z_b^{t_b}), \thinspace\thinspace\thinspace(1\leq j \leq b-1), \\
&w_T(bn+1|z_1\ldots,z_b) = z_bw_T(n|z_1^{t_1}\ldots,z_b^{t_b}) + w_T(n+1|z_1^{t_1}\ldots,z_b^{t_b}).
\end{align*}
These were recently introduced by Dilcher and Ericksen \cite{Karl1}. When $z_1=\cdots =z_b=1,$ we have the reduction $w_T(n+1|1,\ldots, 1)=s_b(n)$. In the case $b=2$, with certain forms of  $t_i$ and $z_i$, we recover polynomial analogues of the Stern sequence introduced by Klav$\v{z}$ar et. al \cite{Klavzar}, Mansour \cite{Mansour}, and Dilcher and Stolarsky \cite{Karl2}. These are the most general Stern polynomials studied to date.

In particular, we introduce a generating product for the base $b$ Stern polynomials and prove an identity for a finite version of this product. We then introduce a representation for Stern polynomials in terms of products of $2\times 2$ matrices, following results for automatic sequences due to Allouche and Shallit \cite{Allouche}. Finally, we characterize the behavior of $w_T$ at indices where $s_b(n)$ is maximized, and generalize some recent continued fractions of Dilcher and Ericksen \cite{Karl3}.

\section{Product representation}
Following from the interpretation of $s(n)$ as counting hyperbinary expansions, we have the generating product
$$\sum_{n=1}^\infty s(n) t^n = t \prod_{i=0}^\infty \left(1+t^{2^i}+t^{2^{i+1}}\right).$$
Interpreted as a product generating partitions, each term in the product, $1+t^{2^i}+t^{2^{i+1}}$, corresponds to the choice of whether to pick no occurrences of the part $2^i$, one occurence of this part, or two. This bijectively generates all hyperbinary expansions of $n$. We can find an analog of this product for the base $b$ Stern polynomials.

\begin{thm}
The base $b$ Stern polynomials have the generating product
\begin{equation}\label{infprod}
 \sum_{n=1}^\infty w_T(n|z_1\ldots,z_b) t^{n}= t\prod_{i=0}^\infty \left(  1+\sum_{j=1}^b z_j^{t_j^i}t^{j\cdot b^i} \right). 
\end{equation}
\end{thm}
\begin{proof}

We consider the simpler case of the hyperternary polynomials ($b=3$), defined by 
\begin{align*}
&w_T(3n-1|x,y,z) = xw_T(n|x^u,y^v,z^w) \\
&w_T(3n|x,y,z) = yw_T(n|x^u,y^v,z^w) \\
&w_T(3n+1|x,y,z) = zw_T(n|x^u,y^v,z^w)+w_T(n+1|x^u,y^v,z^w),
\end{align*}
with $w_T(0) = 0$ and $w_T(1)=1$. Now consider the generating function $$F(t|x,y,z) = \sum_{n=1}^\infty w_T(n|x,y,z) t^{n-1}.$$ We can form a functional equation for $F$ as follows:
\begin{align*}
F(t|x,y,z) &= \sum_{n=1}^\infty w_T(3n-1|x,y,z) t^{3n-2} + \sum_{n=1}^\infty w_T(3n|x,y,z) t^{3n-1} +  \sum_{n=1}^\infty w_T(3n+1|x,y,z) t^{3n} \\
&=\sum_{n=1}^\infty x w_T(n|x^u,y^v,z^w) t^{3n-2} + \sum_{n=1}^\infty y w_T(n|x^u,y^v,z^w) t^{3n-1} \\
&+  \sum_{n=1}^\infty zw_T(n|x^u,y^v,z^w) t^{3n} +  \sum_{n=1}^\infty w_T(n+1|x^u,y^v,z^w) t^{3n} \\
&= xtF(t^3|x^u,y^v,z^w)+ yt^2F(t^3|x^u,y^v,z^w)+  zt^3F(t^3|x^u,y^v,z^w) +  F(t^3|x^u,y^v,z^w) \\
&= (1+xt+yt^2+zt^3)F(t^3|x^u,y^v,z^w).
\end{align*}
We then iterate this functional equation while noting that for $|t|$ sufficiently small, $F(t^n|x,y,z) \to 1$ as $n\to \infty$. This is a subtle fact; if each $|z_i| < 1$ then by the triangle inequality we have $|w_T(n+1|z_1,\ldots,z_b)|<|w_T(n+1|1,\ldots,1)|$, which counts the number of hyper $b$-ary expansions of $n$. We then have a (difficult) characterization of the maximal order \cite[Theorem 1]{Coons}\cite{Defant} $$ \lim \sup_{n\to \infty} \frac{s_b(n)}{n^{\log_b \phi}} = \frac{\phi^{\log_b(b^2-1)}}{\sqrt{5}},$$
where $\phi=\frac12 \left(1 +\sqrt{5}\right)$ is the golden ratio. This shows that $s_b(n) = O(n^{\log_b \phi})$ grows polynomially, so that the generating function $F(t^n|x,y,z) $ will converge to its constant term of $1$ for sufficiently small $t$. Altogether, this gives us the infinite product representation
$$F(t|x,y,z) = \prod_{i=0}^\infty \left(  1+x^{u^i}t^{3^i}+y^{v^i}t^{2\cdot 3^i}+z^{w^i}t^{3\cdot3^i} \right). $$
The arbitrary base analog is proved exactly analogously. We omit the proof because it offers no new insight; the hyperternary case is much more obvious and notationally clearer. 
\end{proof}
\begin{rem}
Note that basically any polynomial generalization of hyperternary expansions begins with the infinite product and inserts variables somewhere inside the product, so that this product representation is excellent motivation for how to define a polynomial version of the base $b$ Stern sequence. This phenomenon is apparent for $b=2$; despite having very different recursive definitions, any form of base $2$ Stern polynomial that has been introduced thus far has a compact generating product.
\end{rem}

Using this infinite product as a starting point, we can derive several properties of the base $b$ Stern polynomials which mirror the base $2$ case. The advantage of this representation is that it gives us some intuition for exactly which theorems are possible to generalize, which isn't very apparent from the recursive definition.

\begin{table}[]
\begin{tabular}{|l|l|l|l|l|l|}
\hline
$n$ & $\omega_T(n|x,y,z)$ & $n$ & $\omega_T(n|x,y,z)$          & $n$ & $\omega_T(n|x,y,z)$                  \\ \hline
1   & $1$                 & 10  & $x^{u^2}+y^vz+z^w$           & 19  & $x^{u^2}y^vz+x^{u^2}z^w+y^{v^2}$     \\
2   & $x$                 & 11  & $x^{1+u^2}+xz^w$             & 20  & $x^{1+u^2}z^w+xy^{v^2}$              \\
3   & $y$                 & 12  & $x^{u^2}y+yz^w$              & 21  & $x^{u^2}yz^w+y^{1+v^2}$              \\
4   & $x^u+z$             & 13  & $x^{u+u^2}+x^{u^2}z+z^{1+w}$ & 22  & $x^{u^2}z^{1+w}+x^uy^{v^2}+y^{v^2}z$ \\
5   & $x^{1+u}$           & 14  & $x^{1+u+u^2}$                & 23  & $x^{1+u}y^{v^2}$                     \\
6   & $x^uy$              & 15  & $x^{u+u^2}y$                 & 24  & $x^uy^{1+v^2}$                       \\
7   & $x^uz+y^v$          & 16  & $x^{u+u^2}z+x^{u^2}y^v$      & 25  & $x^uy^{v^2}z+y^{v+v^2}$              \\
8   & $xy^v$              & 17  & $x^{1+u^2}y^v$               & 26  & $xy^{v+v^2}$                         \\
9   & $y^{1+v}$           & 18  & $x^{u^2}y^{1+v}$             & 27  & $y^{1+v+v^2}$                        \\ \hline
\end{tabular}
\caption{Small $n$ hyperternary polynomials \cite[Table 2]{Karl1}}
\label{Table1}
\end{table}

Another nice consequence is that we can recover the main theorem of \cite{Karl1}, essentially by inspection. We can check the following result with Table \ref{Table1} of small hyperternary polynomials. For what follows, let $\mathbb{H}_{b,n}$ denote the set of hyper $b$-ary expansions of $n\geq 1$. 
\begin{thm}\cite[Theorem 11]{Karl1}
For any integer $n\geq 1$ we have $$w_T(n+1|z_1,\ldots,z_b)  = \sum_{h\in \mathbb{H}_{b,n}} z_1^{p_{h,1}(t_1)}\cdots  z_b^{p_{h,b}(t_b)},$$
where for each $h$ in $\mathbb{H}_{b,n}$, the exponents $p_{h,1}(t_1),\ldots, p_{h,b}(t_b)$ are polynomials in $t_1,\ldots,t_b,$ respectively, with only $0$ and $1$ as coefficients. Furthermore, if for $1\leq j\leq b$ we write 
\begin{equation}\label{thm11.1}
p_{h,j}(t_j) = t_j^{\tau_j(1)}+ t_j^{\tau_j(2)}+\cdots +  t_j^{\tau_j(\nu_j)}, \thinspace\thinspace 0\leq \tau_j(1) < \cdots < \tau_j(\nu_j), \thinspace\thinspace \nu_j \geq 0,
\end{equation}
then the powers that are used exactly $j$ times in the hyper $b$-ary representation of $n$ are 
\begin{equation}
\label{thm11.2}b^{\tau_j(1)},b^{\tau_j(2)},\ldots,b^{\tau_j(\nu_j)}.
\end{equation}
If $\nu_j=0$ in \eqref{thm11.1}, we set $p_{h,j}(t_j)=0$ by convention, and accordingly \eqref{thm11.2} is the empty set.
\end{thm}
\begin{proof}
Consider the infinite product \eqref{infprod}. If we have a hyper $b$-ary expansion of $n$ of the form $n = a_0 + a_1 b+ a_2b^2+\cdots$, where each $0\leq a_i\leq b$, then this contributes a term of the form $\prod_i z_{a_i}^{t^i_{a_i}}$ to the $(n+1)$th Stern polynomial. Summing over all hyperbinary expansions of $n$ completes the proof.
\end{proof}

We can use this product to give a base $b$ polynomial analog of the classic result $$\sum_{n=2^N+1}^{2^{N+1}}s(n) = 3^N ,$$
which was then generalized to base $2$ Stern polynomials \cite[Proposition 5.2]{Karl3} as (in our notation)
$$\prod_{i=0}^{N-1} \left( 2 + z^{t^i}\right) = \sum_{n=2^N+1}^{2^{N+1}}w_T(n|1,z).$$ Our direct proof is combinatorial and quite different from that of \cite{Karl3}, which inductively depends on recursive properties of the Stern polynomials. The fact that we have $z_1=1$ is in fact essential for this identity; inspecting small hyperternary polynomials shows that there is no closed form analog in terms of $\{z_1,\ldots,z_b\}$.
\begin{thm}
Let $l_N:= \frac{b^{N+1}-1}{b-1}$. Then we have the finite product 
$$ \prod_{i=0}^{N-1} \left( 2 +\sum_{j=2}^b z_j^{ b^i} \right) = \sum_{n=l_N-b^N+2}^{l_N+1}w_T(n|1,z_2,z_3\ldots,z_b). $$
\end{thm}
\begin{proof}
As before, let $\mathbb{H}_{b,n}$ denote the set of hyper $b$-ary expansions of $n$. Now consider the finite generating product
$$ P:=\prod_{i=0}^{N-1} \left(  1+\sum_{j=1}^b t^{j\cdot b^i} \right),$$ 
 and let $\mathcal{H}$ denote the set of hyper $b$-ary expansions generated by this product. We then have a bijection between $\mathcal{H}$ and $\cup_{n = l_N-b^N+1}^{l_N} \mathbb{H}_{b,n}$. Furthermore, this bijection preserves all parts of multiplicity $\geq 2$ in a hyper $b$-ary expansion.

First, note that $\mathcal{H}$ contains every hyper $b$-ary expansion in $\cup_{n =0}^{l_N} \mathbb{H}_{b,n}$, except those which contain the part $b^N$ with multiplicity $1$. This is because $P$ generates hyperbinary expansions into parts $\leq b^{N-1}$, so that the largest hyper $b$-ary expansion it can generate is for $l_N$. However, we have the crucial inequality $l_N -b^N < b^N< l_N$, which also implies $l_N<2b^N$, so $b^N$ cannot appear with multiplicity $2$. Therefore, we consider every hyper $b$-ary expansion for $0\leq n\leq l_N-b^N+1$ in $\mathcal{H}$ and append the part $b^N$ to it with multiplicity $1$. This converts it to a hyper $b$-ary expansion for some $ l_N-b^N\leq n \leq l_N $, and we can reverse the mapping by subtracting $b^N$ from any hyper $b$-ary expansion in $\cup_{n = l_N-b^N+1}^{l_N} \mathbb{H}_{b,n}$ in which it appears. This trivially preserves parts of multiplicity $\geq 2$.

Now, set $t=1$ and consider the product
$$ \prod_{i=0}^{N-1} \left( 2 +\sum_{j=2}^b z_j^{ b^i} \right).$$
This generates a sum over $\mathcal{H}$, where each monomial keeps track of every part in a hyper $b$-ary expansion of multiplicity $2$. The right hand side,
$$\sum_{n=l_N-b^N+2}^{l_N+1}w_T(n+1|1,z_2,z_3\ldots,z_b),$$
generates a sum over $\cup_{n = l_N-b^N+1}^{l_N} \mathbb{H}_{b,n}$ which also only tracks parts of multiplicity $\geq 2$. Hence, by our bijection they must be equal.
\end{proof}
Even when $z=1$, this identity for the number of hyper $b$-ary expansions appears to be new: $$(b+1)^N = \sum_{n=l_N-b^N +1}^{l_N} s_b(n).$$

\section{Matrix Representation}
We can also describe an explicit matrix characterization of the Stern polynomials. This is a variant of Allouche and Shallit's famous result \cite{Allouche} about $k$-regular sequences, a generalization of automatic sequences. There are several equivalent definitions of $k$-regular sequences, the most basic of which is that the sequence $s(n)$ is $k$-regular if the \textit{$k$-kernel} $\{(s(k^en+r))_{n\geq 0} | e\geq 0 \text{ and } 0 \leq r \leq k^e-1\}$ generates a finite dimensional vector space over $\mathbb{Q}$. 
\begin{thm} \cite{Allouche}
The integer sequence $\{s(n)\}$ is $k$-regular if and only if there exists positive integers $m$ and $d$, matrices $\mathbf{A}_0,\ldots,\mathbf{A}_{b-1} \in \mathbb{Z}^{d\times d}$, and vectors $\mathbf{v}, \mathbf{w} \in \mathbb{Z}^d$ such that $$s(n) = \mathbf{w}^T\mathbf{A}_{i_0}\cdots \mathbf{A}_{i_s}\mathbf{v}, $$
where $i_s\cdots i_0$ is the base $b$ expansion of $n$.
\end{thm}
The base $b$ Stern sequence satisfies this characterization \cite{Coons} with $$
\mathbf{w} =\begin{pmatrix} 
1 \\
0
\end{pmatrix}
, 
\mathbf{v} =\begin{pmatrix} 
0 \\
1
\end{pmatrix}
, \mathbf{A}_0 =\begin{pmatrix} 
1 & 0 \\
1 & 1
\end{pmatrix}
, \mathbf{A}_1 =\begin{pmatrix} 
1 & 1 \\
0 & 1
\end{pmatrix}
, \mathbf{A}_i =\begin{pmatrix} 
0 & 1\\
0 & 1
\end{pmatrix}
, \thinspace 2 \leq i \leq b-1.$$
We can derive an analog of this result for Stern polynomials, at the cost of parametrizing the matrices. 
\begin{thm}
Let $i_s\cdots i_0$ be the base $b$ expansion of $n$. Then
 $$w_T(n|z_1\ldots, z_b) = \begin{pmatrix} 1 & 0\end{pmatrix}\mathbf{A}_{i_0}(0)\cdots \mathbf{A}_{i_s}(s)\begin{pmatrix} 0 \\1\end{pmatrix}, $$
with $$\mathbf{A}_0(d) =\begin{pmatrix} 
z_{b-1}^{t_{b-1}^d} & 0 \\
z_{b}^{t_{b}^d} & 1
\end{pmatrix}
, \mathbf{A}_1(d) =\begin{pmatrix} 
z_{b}^{t_{b}^d} & 1 \\
0 & z_{1}^{t_{1}^d}
\end{pmatrix}
, \mathbf{A}_i(d) =\begin{pmatrix} 
0 & z_{i-1}^{t_{i-1}^d}\\
0 & z_{i}^{t_{i}^d}
\end{pmatrix}
, \thinspace 2 \leq i \leq b-1.$$
\end{thm}
\begin{proof}
We proceed by induction on the value of $s$, and will show the even stronger statement that $$\mathbf{A}_{i_0}(0)\cdots \mathbf{A}_{i_s}(s) = \begin{pmatrix} * & w_T(n|z_1,\ldots,z_b) \\ * & w_T(n+1|z_1,\ldots,z_b)\end{pmatrix},$$
where asterisks denote stuff we don't care about. Our base case is when $n$ is a single digit, and can easily be verified since we consider $\mathbf{A}_i(0)$, $0\leq i \leq b-1$.

We inductively assume that the theorem holds for all digit strings of length $\leq s+1$, and consider any string of length $s+1$, $i_{s}\cdots i_0,$ which represents the integer $n$. Then consider a digit string of length $s+2$, which we assume to be $i_{s+1}i_{s}\cdots i_0$, and which represents $bn+i_0$. By shifting the parametrization by $1$, we've mapped $z_i \mapsto z_i^{t_i}$, yielding
\begin{align*}
\mathbf{A}_{i_0}(0) \left[\mathbf{A}_{i_1}(1)\cdots \mathbf{A}_{i_{s+1}}(s+1) \right]=  \mathbf{A}_{i_0}(0)  \begin{pmatrix} * & w_T(n|z_1^{t_1},\ldots,z_b^{t_b}) \\ * & w_T(n+1|z_1^{t_1},\ldots,z_b^{t_b})\end{pmatrix}.
\end{align*}
 We then have $3$ cases:
\begin{enumerate}
\item $i_0 = 0$: 
\begin{align*}
 \begin{pmatrix} z_{b-1} & 0 \\z_{b} & 1\end{pmatrix}  \begin{pmatrix} * & w_T(n|z_1^{t_1},\ldots,z_b^{t_b}) \\ * & w_T(n+1|z_1^{t_1},\ldots,z_b^{t_b})\end{pmatrix} &=  \begin{pmatrix} * &  z_{b-1} w_T(n|z_1^{t_1},\ldots,z_b^{t_b}) \\ * &  z_{b} w_T(n|z_1^{t_1},\ldots,z_b^{t_b}) + w_T(n+1|z_1^{t_1},\ldots,z_b^{t_b})\end{pmatrix} \\
&=  \begin{pmatrix} * & w_T(bn|z_1,\ldots,z_b) \\ * & w_T(bn+1|z_1,\ldots,z_b)\end{pmatrix}.
\end{align*}
\item $i_0 = 1$: 
\begin{align*}
 \begin{pmatrix} z_{b} & 1 \\0 & z_1\end{pmatrix}  \begin{pmatrix} * & w_T(n|z_1^{t_1},\ldots,z_b^{t_b}) \\ * & w_T(n+1|z_1^{t_1},\ldots,z_b^{t_b})\end{pmatrix} &=  \begin{pmatrix} * &  z_{b} w_T(n|z_1^{t_1},\ldots,z_b^{t_b}) + w_T(n+1|z_1^{t_1},\ldots,z_b^{t_b})\\ * &  z_{1} w_T(n+1|z_1^{t_1},\ldots,z_b^{t_b}) \end{pmatrix} \\
&=  \begin{pmatrix} * & w_T(bn+1|z_1,\ldots,z_b) \\ * & w_T(bn+2|z_1,\ldots,z_b)\end{pmatrix}.
\end{align*}
\item $2\leq i_0 \leq b-1$: 
\begin{align*}
 \begin{pmatrix} 0 & z_{i_0-1} \\ 0 & z_{i_0}\end{pmatrix}  \begin{pmatrix} * & w_T(n|z_1^{t_1},\ldots,z_b^{t_b}) \\ * & w_T(n+1|z_1^{t_1},\ldots,z_b^{t_b})\end{pmatrix} &=  \begin{pmatrix} * &  z_{i_0-1} w_T(n+1|z_1^{t_1},\ldots,z_b^{t_b}) \\ * &  z_{i_0} w_T(n+1|z_1^{t_1},\ldots,z_b^{t_b}) \end{pmatrix} \\
&=  \begin{pmatrix} * & w_T(bn+i_0|z_1,\ldots,z_b) \\ * & w_T(bn+i_0+1|z_1,\ldots,z_b)\end{pmatrix}.
\end{align*}
\end{enumerate}
Since in each case, we recover $bn+i_0$ and $bn+i_0+1$, we're done.
\end{proof}
For example, consider the hyperternary polynomials again, with $(z_1,z_2,z_3) \mapsto (x,y,z)$ and $(t_1,t_2,t_3) \mapsto (r,s,t)$ for notational clarity. Consider $n=7$, which is represented by the digit string $21$ in base $3$. Then
$$\mathbf{A}_1(0)\mathbf{A}_2(1) =  \begin{pmatrix} z & 1 \\0 & x\end{pmatrix}  \begin{pmatrix} 0 & x^r \\ 0 & y^s\end{pmatrix} = \begin{pmatrix} 0 & x^rz+y^s \\ 0 & xy^s\end{pmatrix} = \begin{pmatrix} 0 & w_T(7|x,y,z) \\ 0 & w_T(8|x,y,z)\end{pmatrix}, $$
just as we expect. To the best of our knowledge, only one such matrix characterization of a polynomial analog of the Stern sequence has previously been stated in the literature \cite{Spiegelhofer}, despite the fact that it is integral to considering the maximal order of the base $b$ Stern sequence \cite{Coons}. Spiegelhofer used this matrix representation to prove a digit reversal property for base $2$ Stern polynomials; however, an arbitrary base analog has been elusive.

\section{Behavior At Maximal Indices}
The maximal values of $s_b(n)$ have been well characterized \cite{Coons}. We can study the related values of $w_T$ at these indices, which will provide polynomial dissections of the Fibonacci numbers. Let $F_k=F_{k-1}+F_{k-2}$ denote the $k$th Fibonacci number, with the initial values $F_0=0, F_1=1$. Let $(a_1a_2\cdots a_l)_b$ denote the integer $a_1 b^{l-1} + a_2b^{l-2}+\cdots a_l$, which is the number read off in base $b$. We also let $((a)^kb)_b = ( \underbrace{a\cdots a }_{\text{$k$ times}}b)_b $. Note that much of the material of this section generalizes recent results of Dilcher and Ericksen \cite{Karl4}. and that by following the methods of that paper we should be able to describe other new finite and infinite continued fractions involving base $b$ Stern polynomials.
\begin{lem}\cite{Defant}
Let $k\geq 2$. Then
$$  \max_{b^{k-2}\leq n < b^{k-1}}s_b(n)= F_k. $$
Moreover, if $a_k$ denotes the smallest $n$ in the interval $[b^{k-2},b^{k-1})$ for which this maximum is attained, then
$$ a_k = \frac{b^k-1}{b^2-1}+\left(\frac{1-(-1)^k}{2}\right)\frac{b}{b+1}.$$
\end{lem}
We will heavily depend on the base $b$ expansions $a_{2l}=((10)^{l-1}1)_b$ and $a_{2l+1} = ((10)^{l-1}11)_b$. In base $b=2$, these are the well known \textbf{Jacobsthal numbers} $\frac{1}{3}(2^n-(-1)^n)$.

We first note some simple properties of this sequence of maximal indices. This first result generalizes the simple recurrence $J_n=2J_{n-1}+(-1)^n$.
\begin{lem}
The base $b$ Jacobsthal numbers satisfy the recurrence
$$a_n=ba_{n-1}+1 -\frac{b}{2}(1+(-1)^n). $$
\end{lem}
\begin{proof}
Begin with the base $b$ expansions of $a_l$ and peel off the last digit to show
\begin{align*}
&a_{2l} = (1-b)+ba_{2l-1}\\
&a_{2l+1} = 1 + ba_{2l}.
\end{align*}
Rewriting the recurrence with a $1+(-1)^n$ indicator to account for the parity of $n$ completes the proof.
\end{proof}

We then study the values of the base $b$ Stern polynomials at these indices. In the $z_i=1$ limit we will recover Fibonacci numbers, so we are essentially studying some renormalized variant of Fibonacci polynomials.

\begin{thm}\label{thm-rec}
We have the recurrences 
\begin{align*}
 &w_T(a_{2l+1}|z_1,\ldots,z_b)  =  z_bw_T( a_{2l} |z_1^{t_1},\ldots,z_b^{t_b}) + z_1^{t_1}  w_T(a_{2l-1}  |z_1^{t_1^2},\ldots,z_b^{t_b^2}),\\ 
 &w_T(a_{2l+2}|z_1,\ldots,z_b)  =  w_T( a_{2l+1} |z_1^{t_1},\ldots,z_b^{t_b}) +z_b z_{b-1}^{t_{b-1}}  w_T(a_{2l}  |z_1^{t_1^2},\ldots,z_b^{t_b^2}). \\ 
\end{align*}

\end{thm}
\begin{proof}
We will now work directly with the digit expansions of $a_n$ and apply the recursive definitions of the Stern polynomials. In particular, assuming $b\geq 3$, we have
\begin{align*}
 w_T(a_{2l+1}|z_1,\ldots,z_b)  &= w_T( ((10)^{l-1}11)_b  |z_1,\ldots,z_b) \\
 &= z_bw_T( ((10)^{l-1}1)_b  |z_1^{t_1},\ldots,z_b^{t_b}) +   w_T( ((10)^{l-2}102)_b  |z_1^{t_1},\ldots,z_b^{t_b}) \\
 &=  z_bw_T( a_{2l} |z_1^{t_1},\ldots,z_b^{t_b}) + z_1^{t_1}  w_T( ((10)^{l-2}11)_b  |z_1^{t_1^2},\ldots,z_b^{t_b^2}) .\\
\end{align*}
If $b=2$ we have $ ((10)^{l-2}102)_b =  ((10)^{l-2}110)_b$ at the second step, but since $z_1 = z_{b-1}$ we obtain the same final recurrence. We perform a similar calculation with $a_{2l+2} = ((10)^l1)_b$ to obtain the second recurrence.
\end{proof}
Interestingly, this recurrence shows that we can set $z_i=0, 2 \leq i \leq b-2,$ in the base $b$ Stern polynomials without affecting their values. This is equivalent to saying that every hyperbinary expansion of $a_n$ is formed of parts with multiplicity $1, b-1,$ or $b$.

Since we have a three term recurrence, we can now write down a simple ``two level" continued fraction involving a ratio of Stern polynomials. If we set $z_i=1$ the limit of this ratio is $\lim_{k\to \infty}F_k / F_{k-1} = \frac12 (1+\sqrt{5})$, the golden ratio. Therefore our continued fraction can be regarded as a polynomial generalization of the canonical continued fraction
$$\phi :=  \frac12 (1+\sqrt{5})  =  1 +\cfrac{1}{1+ \cfrac{1}{1+ \cfrac{1}{1+ \ddots}} }.$$

\begin{thm}
We have the continued fractions
\begin{equation*}
\frac{w_T(a_{2l+1}|z_1,\ldots,z_b)}{w_T( a_{2l} |z_1^{t_1},\ldots,z_b^{t_b})} = z_b + \cfrac{z_1^{t_1}}{
1 + \cfrac{z_b^{t_b}z_{b-1}^{t_{b-1}^2}}{
\ddots  + \cfrac{z_b^{t_b^{2l-3}} z_{b-1}^{t_{b-1}^{2l-2}}}{z_b^{t_b^{2l-2}} +z_1^{t_1^{2l-2}} } } } 
\end{equation*}
and
\begin{equation*}
\frac{w_T(a_{2l+2}|z_1,\ldots,z_b)}{w_T( a_{2l+1} |z_1^{t_1},\ldots,z_b^{t_b})} = 1 + \cfrac{z_bz_{b-1}^{t_{b-1}}}{
z_b^2 + \cfrac{z_1^{t_1^2}}{
\ddots +  \cfrac{z_b^{t_b^{2l-2}} z_{b-1}^{t_{b-1}^{2l-1}}}{z_b^{t_b^{2l-1}} +z_1^{t_1^{2l-1}} } } } .
\end{equation*}
\end{thm}

\begin{proof}
Begin with the recurrences of Theorem \ref{thm-rec} and divide by the middle term. We obtain 
\begin{equation}\label{eq41}
\frac{w_T(a_{2l+1}|z_1,\ldots,z_b)}{w_T( a_{2l} |z_1^{t_1},\ldots,z_b^{t_b})} = z_b + \cfrac{z_1^{t_1}}{
\left( \cfrac{w_T( a_{2l} |z_1^{t_1},\ldots,z_b^{t_b}) }{w_T(a_{2l-1}  |z_1^{t_1^2},\ldots,z_b^{t_b^2})} \right) } 
\end{equation}
and
\begin{equation}\label{eq42}
\frac{w_T(a_{2l+2}|z_1,\ldots,z_b)}{w_T( a_{2l+1} |z_1^{t_1},\ldots,z_b^{t_b})} = 1 + \cfrac{z_bz_{b-1}^{t_{b-1}}}{
\left( \cfrac{w_T( a_{2l+1} |z_1^{t_1},\ldots,z_b^{t_b}) }{w_T(a_{2l}  |z_1^{t_1^2},\ldots,z_b^{t_b^2})} \right) } .
\end{equation}
Substituting \eqref{eq41} into \eqref{eq42} and vice verse gives
\begin{equation*}
\frac{w_T(a_{2l+1}|z_1,\ldots,z_b)}{w_T( a_{2l} |z_1^{t_1},\ldots,z_b^{t_b})} = z_b + \cfrac{z_1^{t_1}}{
1 + \cfrac{z_b^{t_b}z_{b-1}^{t_{b-1}^2}}{
\left( \cfrac{w_T( a_{2l-1} |z_1^{t_1^2},\ldots,z_b^{t_b^2}) }{w_T(a_{2l-2}  |z_1^{t_1^3},\ldots,z_b^{t_b^3})} \right) } } 
\end{equation*}
and
\begin{equation*}
\frac{w_T(a_{2l+2}|z_1,\ldots,z_b)}{w_T( a_{2l+1} |z_1^{t_1},\ldots,z_b^{t_b})} = 1 + \cfrac{z_bz_{b-1}^{t_{b-1}}}{
z_b^2 + \cfrac{z_1^{t_1^2}}{
\left( \cfrac{w_T( a_{2l} |z_1^{t_1^2},\ldots,z_b^{t_b^2}) }{w_T(a_{2l-1}  |z_1^{t_1^3},\ldots,z_b^{t_b^3})} \right) }  } .
\end{equation*}
We can now avoid mixing recurrences, and iterate this construction while noting the base cases $w_T(a_{2}|z_1,\ldots,z_b) = 1$ and $w_T(a_{3}|z_1,\ldots,z_b) = z_b+z_1$.
\end{proof}
These give extremely fastly converging approximations for $w_T$ at maximal indices, since $z_i^{t_i^j} \to 0$ extremely quickly for small $z_i$.

\section{Acknowledgements}
Tanay would like to thank Karl Dilcher for his endless patience, useful discussions, and invitation to Dalhousie. Many thanks also go out to Christophe Vignat, Larry Washington, and especially Wiseley Wong for thoroughly reading various versions of this work.


\begin{thebibliography}{1}
\bibitem{Allouche}J.-P. Allouche and J. Shallit, Automatic
Sequences, Theory, Applications, Generalizations. Cambridge University
Press (2003)

\bibitem{Churchhouse}R. F. Churchhouse, Congruence properties of the binary partition function, {Proc. Cambridge Philos. Soc.}, \textbf{66}, 371--376 (1969)

\bibitem{Coons}M. Coons and L. Spiegelhofer, The maximal order of hyper-({$b$}-ary)-expansions, {Electron. J. Combin.}, \textbf{24: 1}, Paper 1.15, 1--8 (2016)

\bibitem{Defant}C. Defant, Upper bounds for {S}tern's diatomic sequence and related
              sequences, {Electron. J. Combin.}, \textbf{23: 4}, Paper 4.8, 1--47 (2016)

\bibitem{Karl3}K. Dilcher and L. Ericksen, Hyperbinary expansions and {S}tern polynomials, {Electron. J. Combin.}, \textbf{22: 2}, Paper 2.24, 1--18 (2015)

\bibitem{Karl4}K. Dilcher and L. Ericksen, Continued fractions and {S}tern polynomials, {Ramanujan J.}, \textbf{45: 3}, 659--681 (2018)

\bibitem{Karl1}K. Dilcher and L. Ericksen, Polynomials characterizing hyper $b$-ary representations,
Journal of Integer Sequences, Vol. 21, Article 18.4.3 (2018)

\bibitem{Karl2}K. Dilcher and K. Stolarsky, A polynomial analogue to the {S}tern sequence, {Int. J. Number Theory}, \textbf{3: 1}, 85--103 (2007)

\bibitem{Klavzar}S. Klav\v{z}ar, U. Milutinovi\'c, and C. Petr, Stern polynomials, {Adv. in Appl. Math.}, \textbf{39: 1}, 86--95 (2007)


\bibitem{Mansour}T. Mansour, {{$q$}-{S}tern polynomials as numerators of continued
              fractions}, {Bull. Pol. Acad. Sci. Math.}, \textbf{63: 1}, 11--18 (2015)

\bibitem{Spiegelhofer}L. Spiegelhofer, {A digit reversal property for {S}tern polynomials}, {Integers}, \textbf{17}, Paper A53 (2017)






\end{thebibliography}
\end{document}